\documentclass[12pt,a4paper]{amsart}

\usepackage{amsrefs}
\usepackage{hyperref}
\usepackage[scheme=plain]{ctex} 

\DefineSimpleKey{bib}{primaryclass}{}
\DefineSimpleKey{bib}{archiveprefix}{}
\BibSpec{arXiv}{%
  +{}{\PrintAuthors}{author}
  +{,}{ \textit}{title}
  +{}{ \parenthesize}{date}
  +{,}{ arXiv }{eprint}
  +{,}{ primary class }{primaryclass}
}

\usepackage{ctex}
\usepackage{epsfig}
\usepackage{graphicx}
\usepackage{amssymb, amstext, amscd, amsmath}
\usepackage{amsthm, mathrsfs, amsfonts,dsfont}
\usepackage{fullpage}
\usepackage{txfonts}
\usepackage{fancybox}
\usepackage{color}
\usepackage{cite}
\usepackage{comment}

\newtheorem{theorem}{Theorem}[section]
\newtheorem{lemma}[theorem]{Lemma}
\newtheorem{proposition}[theorem]{Proposition}
\newtheorem{corollary}[theorem]{Corollary}
\theoremstyle{definition}
\newtheorem{definition}[theorem]{Definition}
\newtheorem{example}[theorem]{Example}

\newtheorem{conjecture}[theorem]{Conjecture}

\newtheorem{question}[theorem]{Question}

\theoremstyle{remark}

\numberwithin{equation}{section}
\newtheorem*{thm}{\bf{Theorem}}
\newtheorem*{remark}{\bf{Remark}}
\newtheorem*{acknowledgement}{\bf{Acknowledgement}}
\newtheorem*{ThmA}{\bf{Theorem A}}
\newtheorem*{ThmB}{\bf{Theorem B}}

\begin{document}
\title[Spectral dynamics for the infinite dihedral group and the lamplighter group]%
{Spectral dynamics for the infinite dihedral group and the lamplighter group }
\author{Chao Zu}
\address{Department of Mathematics Sciences, Dalian University of Technology,
Dalian, Liaoning, 116024, P. R. China}
\email{zuchao@mail.dlut.edu.cn}

\author{Yixin Yang}
\address{Department of Mathematics Sciences, Dalian University of Technology,
Dalian, Liaoning, 116024, P. R. China}
\email{yangyixin@dlut.edu.cn}

\author{Yufeng Lu}
\address{Department of Mathematics Sciences, Dalian University of Technology,
Dalian, Liaoning, 116024, P. R. China}
\email{lyfdlut@dlut.edu.cn}


\thanks{This research is supported by National Nature Science Foundation of China
(No. 12031002,
11971086).  This research is also partially supported by Dalian High-level Talent innovation Project (Grant 2020RD09) and  Scientific Research Innovation Team of DUT (DUT2021TB07).}
\begin{abstract}
 For a tuple $A=(A_0,A_1,\cdots,A_n)$ of elements in a Banach algebra $\mathfrak{B}$, its projective (joint) spectrum $p(A)$ is the collection of $z\in \mathbb{P}^n$ such that $A(z)=z_0A_0+z_1A_1+\cdots+z_nA_n$ is not invertible. If $\mathfrak{B}$ is the group $C^*$-algebra for a discrete group $G$ generated by $A_0, A_1,\dots, A_n$ with a representation $\rho$, then $p(A)$ is an invariant of (weak) equivalence for $\rho$. In \cite{BY}, B. Goldberg and R. Yang proved that the Julia set $\mathcal{J}(F)$ of the induced rational map $F$ for the infinite dihedral group $D_\infty$ is the union of the projective spectrum with the extended indeterminacy set. But the extended indeterminacy set $E_F$ is complicated. To obtain a better relationship between the projective spectrum and the Julia set, by replacing $A_\pi(z)=z_0+z_1\pi(a)+z_2\pi(t)$ with the extended pencil $A_\pi(z)=z_0+z_1\pi(a)+z_2\pi(t)+z_3\pi(at)$, where $\pi$ is the Koopman representation, and using the method of operator recursions, we show that $p(A_\pi)=\mathcal{J}(F).$ Further, we  study the spectral dynamics for the Lamplighter group $\mathcal{L}$, and prove that $\mathcal{J}(Q)=E_Q$, where $Q$ is the rational map associated with $\mathcal{L}$.

\end{abstract}

\maketitle
\section{Introduction}

Let $\mathcal{B}$ be a complex unital Banach algebra and $A = (A_0 ,A_1 ,\cdots ,A_n )$ be a tuple
of linearly independent elements in $\mathcal{B}$. The multiparameter pencil
$A(z) := z_0 A_0 + z_1 A_1 + \cdots + z_n A_n$
is an important subject of study in numerous fields. The notion of projective (joint)
spectrum was introduced by R. Yang in \cite{R.Yang} as follows.
\begin{definition}\label{1.1}
 For a tuple $(A_0 ,A_1 ,\cdots, A_n )$ of elements in a unital Banach algebra
$\mathcal{B}$, its \textit{projective (joint) spectrum} is defined as
$$P(A) = \{ z\in \mathbb{C}^{n+1} ~|~ A(z) ~~is ~~not ~~invertible \}.$$
The projective resolvent set refers to the complement, $P^c(A) = \mathbb{C}^{n+1}\setminus P(A)$.

\end{definition}

The projective spectrum reveals the joint behaviors of $A_0 ,A_1 ,\cdots ,A_n$ as well as interactions among them, and  properties of the projective spectrum have been previously investigated
in a series of papers, such as \cite{BPY,PY,CSZ,DY1,GY,HWY} and  references therein.  One interesting case is when the tuple $A$ is
associated with a group $G$ which has finite generating set $S = \{g_1 , g_2 ,\cdots, g_s \}$ and a unitary representation $\pi$ of $G$ on a complex Hilbert space $\mathcal{H}$. Then the projective spectrum encapsulates information about $G$ as well as some subtler information about the representation. For instance, for the multiparameter pencil
$A_\pi (z) = z_0 I + z_1 \pi(g_1 ) + \cdots + z_s \pi(g_s )$, it was shown in \cite{R.GY}  that if two representations $\pi$ and $\rho$ are weakly equivalent, then $P(A_\pi )=P(A_\rho )$.
Another important study of spectral theory on groups is the group $\mathcal{G}$ of intermediate growth, which was done by Grigorchuk
and his collaborators (cf. \cite{GN,GNS,GNZ,GS}). It was discovered that $\mathcal{G}$ has a self-similar representation on the rooted binary tree, and this self-similarity induces
a rational map on a certain spectral set of $\mathcal{G}$ whose dynamical properties link tightly to spectral properties of $\mathcal{G}$, which is called \textit{spectral dynamics} (cf. \cite{BY}). This paper aims to investigate the spectral dynamics from the view point of projective spectrum, and we will focus on the the infinite dihedral group $D_\infty$ and the lamplighter group $\mathcal{L}$.

Let $\mathbb{P}^n$ be the complex projective space of dimension $n$, and $\phi$ be the canonical projection   from  $\mathbb{C}^{n+1}$ onto $\mathbb{P}^n$. Let $p(A)=\phi(P(A))$ and $p^c(A)=\phi(P^c(A))$, that is,
\[   p(A) = \{ z\in \mathbb{P}^{n} ~|~ A(z) ~~is ~~not ~~invertible \}.   \]
For $D_\infty=<a,t~|~ a^2=t^2=1>$, B. Goldberg and R. Yang considered the pencil $A_\pi(z)=z_0I+z_1\pi(a)+z_2\pi(t)$ in \cite{BY}, where $\pi$ is the Koopman representation of $D_\infty$. By the self-similarity of $\pi$ and a Schur complement argument, they defined a polynomial map $F$ from $\mathds{P}^2$ into $\mathds{P}^2$, which preserves the projective spectrum $p(A_\pi)$. Through studying the dynamical properties of $F$, they showed that $p(A_\pi)\bigcup E_F=\mathcal{J}(F)$, where
\[ p(A_\pi)=\bigcup_{x\in [-1,1]}\{z\in \mathbb{P}^2 : z_0^2-z_1^2-z_2^2-2xz_1z_2=0\}. \]
 However, compare to the concise expression of $p(A_\pi)$, the extended indeterminacy set $E_F$ is complicated. In this paper, we consider the extended pencil $A_\pi(z)=z_0I+z_1\pi(a)+z_2\pi(t)+z_3\pi(at)$. By the operator recursions induced by the self-similar representation $\pi$,  we obtain a polynomial map $F: \mathds{P}^3\rightarrow \mathds{P}^3$, and prove that the projective spectrum $p(A_\pi)$ coincides with the Julia set $\mathcal{J}(F)$.

\begin{ThmA}
For the pencil $A_\pi(z)=z_0+z_1\pi(a)+z_2\pi(t)+z_3\pi(at)$,
%
\[F(z)=[~z_0^2-z_1^2~:~z_0z_2-z_1z_3~:~ z_0z_2-z_1z_3~:~z_2^2-z_3^2~],\]
and
\begin{equation*}
  \mathcal{J}(F)=\bigcup_{x\in [-1,1]}\{z\in \mathbb{P}^3 : z_0^2+z_3^2-z_1^2-z_2^2+2x(z_0z_3-z_1z_2)=0\}=p(A_\pi).
\end{equation*}

\end{ThmA}

The lamplighter group $\mathcal{L}$ is the group
$$(\bigoplus_{\mathbb{Z}}\mathbb{Z}/2\mathbb{Z})\rtimes \mathbb{Z}(= \mathbb{Z}/2\mathbb{Z} \wr \mathbb{Z} ).$$
It is generated by two elements $u, v$, where $u$ is the generator of $\mathbb{Z}$ in the semi-direct product, and
\[v=(\cdots,0,0,1,0,0,\cdots)\in \bigoplus_{\mathbb{Z}}\mathbb{Z}/2\mathbb{Z}\]
with $1$ is in zeroth position.
In \cite{GZ}, R. Grigorchuk and A. Zuk proved that the lamplighter group can be realized as  a 2-state automaton group $G:=<a, b>,$  where $a=(a,b)\sigma, b=(a,b)$ and $\sigma$ is the generator of the symmetric group of order $2$ , via the isomorphism
\begin{align*}
  a & \rightarrow  u \\
  c=b^{-1} a & \rightarrow  v.
\end{align*}

For the Markov operator $M$ associated with a group and a generating set, the spectral measure $\mu$ of $M$ can be decomposed into absolutely continuous, singular continuous, and discrete components, i.e. $\mu=\mu_{ac}+\mu_{sc}+\mu_d$. A spectrum point $x$ is called a pure point if it is only contained in the support of discrete components, that  is, there is a neighborhood $N_x$ of $x$ such that $\mu_d(N_x)>0, \mu_{ac}(N_x)=\mu_{sc}(N_x)=0$.
For the lamplighter group and generating set $\{a,b\}$,  R. Grigorchuk and A. Zuk \cite{GZ} showed that the spectrum of the associated Markov operator $M$ are all pure points.   This is the first example of a group and generating set with pure point spectrum of a Markov operator.  In \cite{GBS}, by a careful study of spectral properties of a one-parametric family $a+a^{-1}+b+b^{-1}-\mu c$ of convolution operators on $\mathcal{L}$, where $\mu$ is a real parameter, R. Grigorchuk and B. Simanek showed  that the lamplighter group $\mathcal{L}$ has a system of generators for which the spectrum of the discrete Laplacian on the Cayley graph is a union of an interval
and a countable set of isolated points accumulating to a point outside this interval. This is the first example of a group with infinitely many gaps in the spectrum of Cayley graph. These interesting examples lead us to study the spectral dynamics of the the lamplighter group $\mathcal{L}$.
\begin{ThmB}
For the lamplighter group $\mathcal{L}$ and the pencil $D_\pi(z)=z_0+z_1\pi(c)+z_2(\pi(a)+\pi(b))+z_3(\pi(a^{-1})+\pi(b^{-1}))$, the associated operator recursion induces a rational map on $\mathds{P}^3$
$$Q(z)=[~z_0+z_1-\frac{2z_2z_3}{z_0-z_1}~:~ \frac{2z_2z_3}{z_0-z_1}~:~ z_2~:~z_3~].$$
For the rational map $Q$, the Julia set $\mathcal{J}(Q)$ coincides with the extended indeterminacy set $E$,
\begin{equation*}
  \mathcal{J}(Q)=E= \{~(z_0-z_1)z_1=2z_2z_3~\}\bigcup \left( \bigcup_{x\in [0,1]} \{~(z_0+z_1)^2-16xz_2z_3=0 ~\}\right) \bigcup \left( \bigcup_{n=0}^\infty \Gamma_n\right),
\end{equation*}
where $\Gamma_n$ is the curve $  ~\{ ~   (z_0-z_1)U_n(\tau(z))=\sqrt{4z_2z_3}U_{n-1}(\tau(z))~\} $ with $\tau(z)=\frac{z_0+z_1}{2\sqrt{4z_2z_3}}$ and $U_n$ is the $n$-th Chebyshev polynomial of the second kind. Moreover, the Julia set $\mathcal{J}(Q)$ is properly contained in the projective spectrum $p(D_\pi)$,
\begin{equation*}
   \phi(L) \cup \mathcal{J}(Q)\subseteq p(D_\pi) ,
\end{equation*}
where $L$ is the hyperplane $\{z_0+z_1+2z_2+2z_3=0 \}$ and $\phi$ is the canonical projection.

\end{ThmB}

The reminder of the paper is organized as follows. In Section 2, we introduce some basic definitions and necessary concepts. In Section 3, we study the spectral dynamics of the infinite dihedral group $D_\infty$ and give the proof of Theorem A. In Section 4, we study the spectral dynamics of the Lamplighter group $\mathcal{L}$ and give the proof of Theorem B.

\section{Preliminaries}
\subsection{Projective spectrum of $D_\infty$ with respect to the left regular representation }

It is proved in \cite{R.GY} that the left regular representation $\lambda$ of $D_\infty$ is equivalent to the following  representation $\rho$ of $D_\infty$ on $L^2(\mathds{T},\frac{d\theta}{2\pi})\bigoplus L^2(\mathds{T},\frac{d\theta}{2\pi})$ defined by:
\begin{equation}\label{2.1}
 \rho(a)=\left(
                  \begin{array}{cc}
                    0 & I \\
                    I & 0 \\
                  \end{array}
                \right),~~ \rho(t)=\left(
                  \begin{array}{cc}
                    0 & T \\
                    T^* & 0 \\
                  \end{array}
                \right),
\end{equation}
where $I$ is the identity operator on $L^2(\mathds{T},\frac{d\theta}{2\pi})$, and $T : L^2(\mathds{T},\frac{d\theta}{2\pi})\rightarrow L^2(\mathds{T},\frac{d\theta}{2\pi})$ is
the bilateral shift, i.e., the unitary operator defined by $Tf(e^{i\theta} ) = e^{i\theta}f(e^{i\theta})$.

 Let
\begin{equation}\label{2.2}
  T=\int_0^{2\pi} e^{i\theta} dE(e^{i\theta})
\end{equation}
be the spectral decomposition of $T$, and $A_\rho(z)=z_0+z_1\rho(a)+z_2\rho(t)+z_3\rho(at)$, then combining together (\ref{2.1}) and (\ref{2.2}), we have
$$A_\rho(z)=\int_0^{2\pi} \left(
                                 \begin{array}{cc}
                                   z_0+z_3e^{i\theta} & z_1+z_2e^{i\theta} \\
                                   z_1+z_2e^{-i\theta} & z_0+z_3e^{-i\theta} \\
                                 \end{array}
                               \right)dE(e^{i\theta}).$$
Hence the pencil $A_\rho(z)$ is not invertible if and only if $\left(
                                 \begin{array}{cc}
                                   z_0+z_3e^{i\theta} & z_1+z_2e^{i\theta} \\
                                   z_1+z_2e^{-i\theta} & z_0+z_3e^{-i\theta} \\
                                 \end{array}
                               \right)$
is not invertible for at least one $\theta \in [0,2\pi)$. Thus we have
\begin{lemma}{\label{lemma2.1}}
 Let $\lambda : D_\infty \rightarrow U(l^2(D_\infty))$ be the left regular representation. Then
 $$P(A_\lambda)=P(A_\rho)=\bigcup_{x\in [-1,1]}\{z\in \mathbb{C}^4 : z_0^2+z_3^2-z_1^2-z_2^2+2x(z_0z_3-z_1z_2)=0\}.$$
\end{lemma}

\subsection{Regular rooted trees and self-similar groups}
Let $X$ be a finite set, usually called the alphabet, of size $k$. The set of all finite
words over $X$ is denoted by $X^*$. The set $X^*$ can be naturally equipped with the
structure of a  $k$-regular rooted tree as follows. The vertices of the tree are the words
in $X^*$, the root is the empty word, the level $n$ is the set $X^n$ of words of length $n$ over $X$, and the children of each vertex $u\in X^*$ are the $k$ vertices of the form $ux$, for $x\in X$. We use $X^*$ to denote the set of finite words over X, the set of vertices of the rooted tree we just described, as well as the tree itself.

The group $Aut(X^*)$ of all automorphisms of the  k-regular rooted tree $X^*$
preserves the root and all levels of the tree. Every automorphism $g\in Aut(X^*)$
induces a permutation $\alpha_g$ of $X$, defined by $\alpha_g(x)=g(x)$, called the root permutation of $g$. It represents the action of $g$ at the first letter in each word. For every automorphism $g\in Aut(X^*)$ and every vertex $u\in X^*$, there exists a unique tree automorphism of $X^*$ , denoted by $g_u$ , such that, for all words $w\in X^*$,
$$g(uw)=g(u)g_u(w).$$
The automorphism $g_u$ is called the section of $g$ at $u$: It represents the action of $g$ on
the tails of words that start with $u$. Every automorphism $g$ is uniquely determined
by its root permutation $\alpha_g$ and the $k$ sections at the first level $g_x$ , for $x\in X$. Indeed,
for every $x\in X$ and $w\in X^*$ we have
$$g(xw)=\alpha_g(x)g_x(w).$$

When $X=\{0,1,\cdots,k-1\}$, a succinct representation, called wreath recursion,
of the automorphism $g\in Aut(X^*)$, describing its root permutation and its first level
sections is given by
$$g=\alpha_g(g_0,g_1,\cdots,g_{k-1}).$$
In addition of being short and clear, it has many other advantages, not the
least of which is that it emphasizes the fact that $Aut(X^*)$ is isomorphic to the
semidirect product $Sym(X)\ltimes (Aut(X^*))^X$, that is, to the permutational wreath
product $Sym(X)\wr_X Aut(X^*)$, where $Sym(X)$ is the group of all permutations of $X$.

A set $S\subset Aut(X^*)$ of tree automorphisms is \textit{self-similar} if it is closed under
taking sections, that is, every section of every element of $S$ is itself in the set $S$.
Thus, for every word $u$, the action of every automorphism $s\in S$ on the tails of
words that start with $u$ looks exactly like the action of some element of $S$. Note that
for a set $S$ to be self-similar it is sufficient that it contains the first level sections of
all of its elements. Indeed, this is because $g_{uv}=(g_u)_v$ , for all words $u,v\in X^*$.

A group $G\leq Aut(X^*)$ of tree automorphisms is \textit{self-similar} if it is self-similar as a set. Every group generated by a self-similar set is itself self-similar. This is because
¡°sections of the product are products of sections¡± and ¡°sections of the inverse are
inverses of sections¡±. To be precise, for all tree automorphisms $g$ and $h$ and all words $u\in X^*$,
\[(gh)_u=g_{h(u)}h_u~~,~~~~~ (g^{-1})_u=(g_{g^{-1}(u)})^{-1}.\]

\subsection{Self-similar representation and weak equivalence}

 For a countable group $G$ and a unitary representation $\rho : G \rightarrow U(\mathcal{H})$,
where $U(\mathcal{H})$ denotes the group of unitary operators on the Hilbert space $\mathcal{H}$.
 The representation $\rho$ is called \textit{self-similar}, or more precisely \textit{$d$-similar}, if there exists
a natural number $d$ and a unitary map $W : \mathcal{H} \rightarrow \mathcal{H}^d$ such that for every $g\in G$, the
$d \times d$ block matrix $X(g) := W\rho(g)W^*$ has all of its entries $x_{ij}$ either equal to 0 or
from $\rho(G)$. Since $X(g)$ itself and each of its nonzero entry
$x_{ij}$ are unitaries, every row or column of $X(g)$ has precisely one nonzero entry.

Self-similar representations always arise when the self-similar group $G$ acts on a $d$-regular rooted tree. An account of the theory
of self-similar group can be found in the surveys \cite{GNS,GNZ,GN} and
in Nekrashevych's book \cite{V.N}. Fix $d \geq 2$ an integer and let $T = T_d$ be the $d$-regular rooted tree whose vertices
are in bijection with finite words (strings) over an alphabet of cardinality $d$ (a standard choice for $A$ is $\{0,\cdots,d-1\}$, see Figure 1 below).
\begin{figure}[htb]

\centering

\includegraphics[height=6.0cm,width=9.5cm]{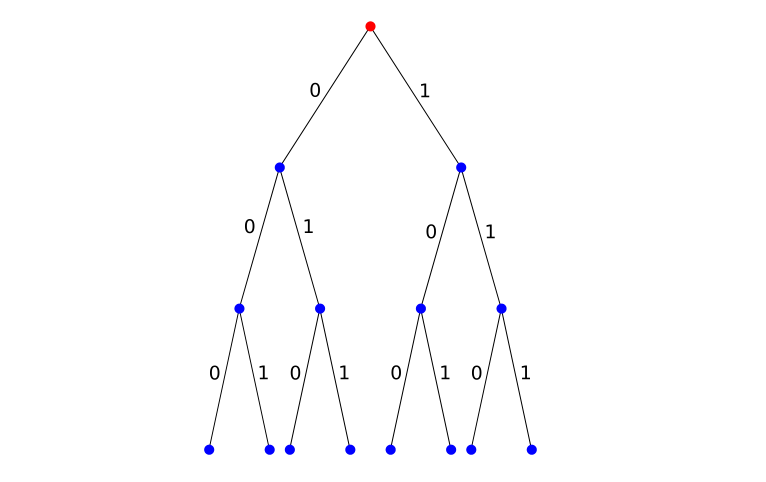}
\caption{2-regular rooted  tree}
\end{figure}
The boundary $\partial T$ of the tree consisting of geodesic paths jointing the root vertex to infinity has a natural topology which makes it homeomorphic to the Cantor set. The mean measure $\mu$ on $\partial T$ is defined by $\mu(\partial  T_{i_1\cdots i_p} ) =\frac{1}{d^p}$, where $i_k\in \{0,\cdots,d-1\}$ for each $k$. In other words, the measure $\mu$ distributes evenly on the subtrees at every level.
A self-similar group naturally acts on a  $d$-regular rooted tree and this
action respects the self-similar structure of the tree. Namely, for each element
$g\in G$ and a vertex $v\in V(T)$, the restriction $g_v$ of $g$ on the subtree $T_v$ rooted
at $v$ can be identified with an element of $G$ (using the canonical identification of $T_v$ with $T$). Define the Hilbert space $\mathcal{H} = L^2 (\partial T,\mu)$. Let $\mu_i = d\cdot \mu ,~ i = 0, 1,\cdots, d-1$ be the normalized restrictions of $\mu$ on the boundary of the subtrees $\partial T_0 , \partial T_1,\cdots, \partial T_{d-1} $, and define $\mathcal{H}_i = L^2 (\partial T_i , \mu_i ), i = 0,1,\cdots, d-1$. Then each $\mathcal{H}_i$ can be identified with $\mathcal{H}$ and hence $\mathcal{H} = \mathcal{H}_0 \bigoplus  \mathcal{H}_1\bigoplus\cdots\bigoplus \mathcal{H}_{d-1}$ can be identified with $\mathcal{H}^d$ by a unitary $W$.

If a locally compact group $G$ has a measure-preserving action on a measure space
$(X,\mu)$, the \textit{Koopman representation} $\pi : G \rightarrow U(L^2(X,\mu))$
is defined by
$$\pi(g)f(x) = f(g^{-1} x), ~~~\forall x\in  X,~~ \forall g \in G.$$
In particular, if $X=G$ and $\mu$ is the Haar measure, then the Koopman representation is exactly the left regular representation of $G$.

Clearly, for a self-similar group $G$, its natural action on $\partial T$  preserves the mean measure $\mu$ on $\partial T$, hence its Koopman representation $\pi : G \rightarrow U(L^2(\partial T,\mu))$ is a self-similar representation. In this paper, we are primarily concerned  the Koopman representation of two groups: the infinite dihedral group $D_\infty$ and the Lamplighter group $\mathcal{L}$ , both of which are self-similar groups acting on the  2-regular rooted  tree.

\begin{definition}\label{2.3}
Consider two unitary representations $\pi$ and $\rho$ of a discrete group $G$
in Hilbert spaces $\mathcal{H}$ and $\mathcal{K}$, respectively. One says that $\pi$ is weakly contained in $\rho$
(denoted by $\pi \prec \rho $) if for every $x\in \mathcal{H}$, every finite subset $F\subset G$, and every $\epsilon> 0$,
there exist $y_1 , y_2 , \cdots , y_n$ in $\mathcal{K}$ such that for all $g \in F$,
$$|\langle \pi(g)x, x\rangle-\sum_{i=1}^n \langle \rho(g)y_i, y_i\rangle|<\epsilon.$$
\end{definition}
Weak containment plays an important role in the representation theory of countable groups. In \cite{R.GY}, it was shown that for a discrete group $G$ and a pencil $A(z)=\sum_{i=0}^k z_i g_i$   in the group algebra $\mathbb{C}[G]$, if $\pi \prec \rho $, then $P(A_\pi)\subset P(A_\rho)$. Two unitary representations $\pi$ and $\rho$ are said to be weakly equivalent if $\pi \prec \rho $ and $\rho \prec \pi $, and we denote this by $\pi \sim \rho$. So the projective spectrum is invariant with respect to the weak equivalence of representations. In particular, for the infinite dihedral group $D_\infty$, let $\lambda$ and $\pi$ be the left regular representation and the Koopman
representation of $D_\infty$, respectively, R. Yang and R. Grigorchuk proved that they are weakly equivalent (see \cite{R.GY}, Theorem 7.2).
So we have $P(A_\lambda ) = P(A_\pi )$ for $D_\infty$, and combine with Lemma \ref{lemma2.1},  we get a partial result of Theorem A:
 $$P(A_\pi)=P(A_\lambda)=P(A_\rho)=\bigcup_{x\in [-1,1]}\{z\in \mathbb{C}^4 : z_0^2+z_3^2-z_1^2-z_2^2+2x(z_0z_3-z_1z_2)=0\}.$$
\subsection{Operator recursions and spectra}

 A group $G$ defined by an automaton over an alphabet of $d$ letters naturally acts by automorphisms on a $d$-regular
rooted tree (see \cite{GNS}). For the $2$-regular rooted tree $T$,  let $\xi_n=\{E_w;~w\in \{0,1\}^n \}$ be a partition of the boundary $\partial T$ into $2^n$ atoms,
where $E_w$ is an open and closed subset consisting of paths starting with the string $w$. Let $\mathcal{H}=L^2(\partial T, \mu)$ and $\mathcal{H}_n$ be a subspace of $\mathcal{H}$ spanned by characteristics functions
of atoms of the partition $\xi_n$ . Then $\dim \mathcal{H}_n=2^n$ and $\mathcal{H}_n$ is invariant with respect to the
Koopman representation $\pi$, because any automorphism of a tree permutes the atoms of the partition $\xi_n$. There is a natural isomorphism $\mathcal{H}_n\simeq \mathcal{H}_{n-1}\bigoplus \mathcal{H}_{n-1}$ by devide $\xi_n$ into $\xi_{n-1}^0=\{E_{0w}; w\in\{0,1\}^{n-1}\}$ and $\xi_{n-1}^1=\{E_{1w}; w\in\{0,1\}^{n-1}\}$.

Let $\pi_n=\pi|_{\mathcal{H}_n}$, then for an  element $m\in \mathbb{C}[G]$ of the group algebra, it is clear that
$$\sigma(\pi(m))\supset \overline{\bigcup_{n\geq 0}\sigma(\pi_n(m))}.$$
It's natural to ask that under what conditions  the equation $\sigma(\pi(m))= \overline{\bigcup_{n\geq 0}\sigma(\pi_n(m))}$  holds?  If the above equation holds, then we can reduce the computation of the projective spectrum to the finite dimension cases. When $\pi(m)$ is a self-adjoint operator, the above equation always holds. For this reason, many real pencils have been widely studied (cf.\cite{GZ,GNZ, DG}).

As an explicit example of operator recursions, assume $G=D_\infty=<a,t>$, after fixing the isomorphism $\mathcal{H}\simeq \mathcal{H}\bigoplus \mathcal{H}$, the operators $\pi(a),\pi(t)$ which are still denoted by $a$ and $t$,  satisfy the following operator recursion:
$$a=\left(
      \begin{array}{cc}
        0 & I \\
        I & 0 \\
      \end{array}
    \right)
,~~t=\left(
       \begin{array}{cc}
         a & 0 \\
         0 & t \\
       \end{array}
     \right),
$$
which corresponds to the wreath type relation: $a=\sigma, t=(a,t)$.

Let $a_n , t_n$ be the matrices corresponding to the generators for the representation $\pi_n, n=0,1,\cdots$, that is, $\pi_n(a),\pi_n(t)$. Then $a_0=b_0=Id$, and
$$a_n=\left(
      \begin{array}{cc}
        0 & I_{n-1} \\
        I_{n-1} & 0 \\
      \end{array}
    \right)
,~~t_n=\left(
       \begin{array}{cc}
         a_{n-1} & 0 \\
         0 & t_{n-1} \\
       \end{array}
     \right),~~~~~~~n\geq 1,
$$
where we keep in mind the natural isomorphism $\mathcal{H}_n\simeq \mathcal{H}_{n-1}\bigoplus \mathcal{H}_{n-1}$ . This recurrent relation
will play the major role in the section 3.

As another explicit example of operator recursions, assume $G=\mathcal{L}\cong <a,b>$, after fixing the isomorphism $\mathcal{H}\simeq \mathcal{H}\bigoplus \mathcal{H}$, the operators $\pi(a),\pi(b)$ which are still denoted by $a$ and $b$, satisfy the following operator recursion:
$$a=\left(
      \begin{array}{cc}
        0 & a \\
        b & 0 \\
      \end{array}
    \right)
,~~b=\left(
       \begin{array}{cc}
         a & 0 \\
         0 & b \\
       \end{array}
     \right),
$$
which corresponds to the wreath type relation: $a=(a,b)\sigma, b=(a,b)$.

Let $a_n , b_n$ be the matrices corresponding to the generators for the representation $\pi_n, n=0,1,\cdots$, that is, $\pi_n(a),\pi_n(b)$. Then $a_0=b_0=Id$, and
$$a_n=\left(
      \begin{array}{cc}
        0 & a_{n-1} \\
        b_{n-1} & 0 \\
      \end{array}
    \right)
,~~b_n=\left(
       \begin{array}{cc}
         a_{n-1} & 0 \\
         0 & b_{n-1} \\
       \end{array}
     \right),~~~~~~~n\geq 1,
$$
where we keep in mind the natural isomorphism $\mathcal{H}_n\simeq \mathcal{H}_{n-1}\bigoplus \mathcal{H}_{n-1}$ . This recursion
will play the major role in the section 4.

\subsection{Some fundamentals of complex dynamics}

Let $\hat{\mathbb{C}}=\mathbb{C}\bigcup \infty$ be the Riemann sphere, and for a function $H :\hat{ \mathbb{C}}\rightarrow \hat{ \mathbb{C}}$, let $H^n$ denote the $n$-th
iteration of $H$. In this paper, this notation will not cause confusion with power function. The complex dynamics studies various issues concerning the convergence of the sequence
$\{H^n \}_{n=1}^\infty$ . We will mention some definitions pertaining to our study and readers can find more information on this subject in \cite{F.B,J.M,T.U}.

\begin{definition}\label{2.4}
 Given a non-constant rational function $H :\hat{ \mathbb{C}}\rightarrow \hat{ \mathbb{C}}$, its Fatou set
$\mathcal{F}(H)$ is the maximal open subset of $\hat{ \mathbb{C}}$
 on which the sequence $\{H^n \}_{n=1}^\infty$ is equicontinuous, and the Julia set $\mathcal{J}(H)$ is the complement $\hat{ \mathbb{C}} \setminus \mathcal{F}(H) $.

\end{definition}
For the  Tchebyshev polynomial $T(z)=2z^2-1$, the following theorem is known (see \cite{F.B}).
\begin{lemma}\label{T}
For the map $T(z) = 2z^2-1$ on $\hat{ \mathds{C}}$, its Julia set $\mathcal{J}(T) =
[-1,1]$, and the iteration sequence $\{T^n \}$ converges normally to $\infty$ on the
Fatou set $\hat{ \mathbb{C}}\setminus [-1,1]$. Further, the zeros of the iteration sequence $\{T^n\}$ is densely contained in $[-1,1]$.
\end{lemma}
Moreover, we consider the dynamics in complex projective space. Let $H:\mathbb{P}^3 \rightarrow \mathbb{P}^3$ be a polynomial map defined by $H(z)=[A(z):B(z):C(z):D(z)]$, where $A,B,C$ and $D$ are homogeneous polynomials in $[z_0:z_1:z_2:z_3]$ of the same degree $d\geq 2$.
In the case that there are only a finite number of such common zeros,
the map $H$ is called a rational map. A point $z\in \mathbb{P}^3$ is called an  indeterminacy point if $A(z) = B(z) = C(z) =D(z)= 0$. The set of all indeterminacy points of $H$ is called the indeterminacy set of $H$, denote by $I(H)$, or simply $I$ when there is no confusion. And the $n$-indeterminacy set $I_n$ of $H$ is defined by $I(H^n)$. In the following, we always use the canonical lift of $H$ to $\mathbb{C}^4$ to calculate $I_n$, and in this case, $I_n=\phi(H^{-n}(O))$, where $H$ is regarded as its canonical lift to $\mathbb{C}^4$.
The extended indeterminacy set of $H$ is defined as
$$E_H=\overline{\bigcup_{n\geq 1}I_n}=\overline{\bigcup_{n\geq 1} \phi(H^{-n}(O)}).$$
For convenience, we write it as $E$ when there is no confusion about the map $H$.
\begin{definition}\label{2.6}
 Let $H : \mathbb{P}^3 \rightarrow \mathbb{P}^3$ be a rational map such that $\mathbb{P}^3 \setminus E$ is nonempty.
A point $p \in \mathbb{P}^3 \setminus E $ is said to be a Fatou point if there exists for every $\epsilon>0$ some
neighborhood $U$ of $p$ such that $\textit{diam} H^n (U\setminus E) <\epsilon$ for all $n$, where \textit{diam} stands
for the diameter with respect to the Fubini-Study metric on $\mathbb{P}^3 $ . The Fatou set $\mathcal{F}(H)$
is the set of Fatou points of $H$. The Julia set $\mathcal{J}(H)$ is the complement $\mathbb{P}^3\setminus \mathcal{F}(H)$.
\end{definition}
It is shown in \cite{T.U} that a point $p$ is in $\mathcal{F}(H)$ if and only if there exists a neighborhood $V$ of $p$ such that the sequence $\{H^n\}$ is a normal family on $V$. By a normal
family argument, it follows that the Fatou set is open and the Julia set is closed. The
extended indeterminacy set $E$ is clearly contained in the Julia set. Lastly, the map $H$ preserves the Julia as well as the Fatou set.

\section{The spectral dynamics of the infinite dihedral group $D_\infty$ }

The following lemma will be used in this section.
\begin{lemma}[\cite{GZ}, Lemma 7]\label{lemma3.1}
 Let $A,B,C$ and $D$ be $n$ by $n$ matrices with complex coefficients such that
$AC =CA$. Then
\begin{equation*}
  \det \left(
         \begin{array}{cc}
           A & B \\
           C & D \\
         \end{array}
       \right)=\det(AD-CB).
\end{equation*}

\end{lemma}
\begin{proof}
Without loss of generality, assume $A$ is invertible, then by the argument of Schur complement, we have
\begin{align*}
  \det \left(
         \begin{array}{cc}
           A & B \\
           C & D \\
         \end{array}
       \right) & =\det(A)\det(D-CA^{-1}B) \\
   & =\det(AD-ACA^{-1}B)=\det(AD-CB).
\end{align*}

\end{proof}

The Koopman representation of $D_\infty$ on the $2$-regular rooted tree $T$ is realized by the following automaton in Figure 2,  where $a$ and $t$
are automorphisms of $T$ satisfying the recursive relation $a=\sigma, t=(a,t)$.
\begin{figure}[htb]

\centering

\includegraphics[height=4.5cm,width=9.5cm]{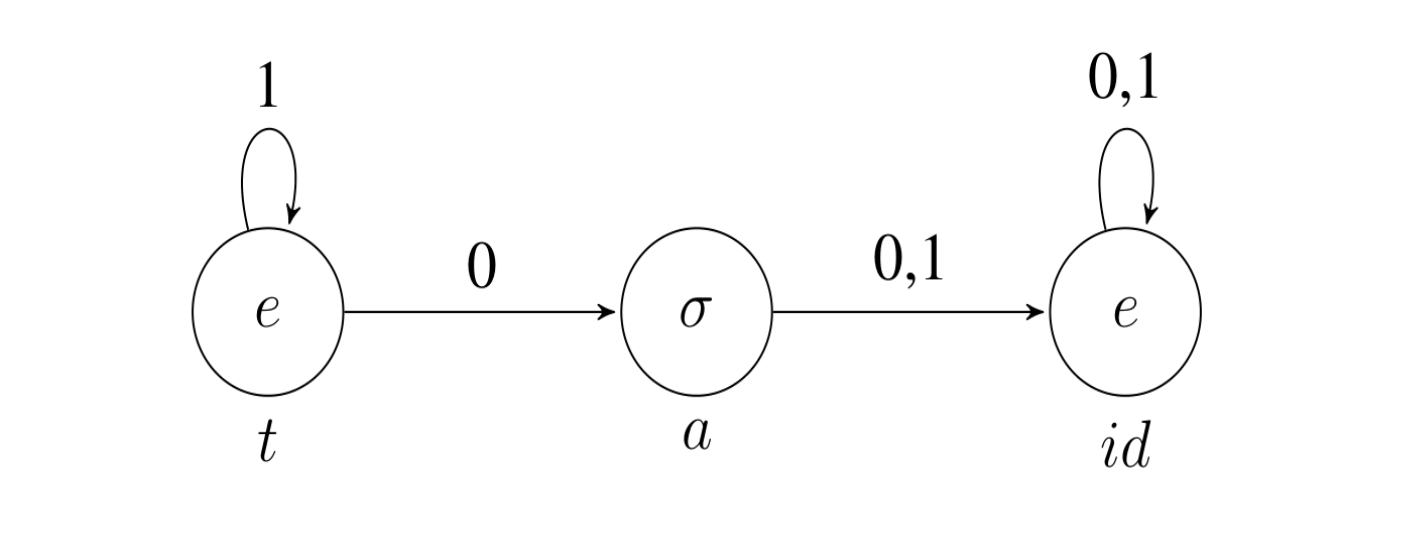}

\caption{Automaton of the group $D_\infty$}

\end{figure}

For convenience, we denote $\pi_n(a), \pi_n(b)$ by $a_n, b_n$ in short, then $a_0=t_0=Id$, and
$$a_n=\left(
        \begin{array}{cc}
          0 & I_{n-1} \\
          I_{n-1} & 0 \\
        \end{array}
      \right)
,~~t_n=\left(
         \begin{array}{cc}
           a_{n-1} & 0 \\
           0 & t_{n-1}\\
         \end{array}
       \right),~~
a_nt_n=\left(
         \begin{array}{cc}
           0 & t_{n-1} \\
           a_{n-1} & 0 \\
         \end{array}
       \right).
$$
Then for $A_{\pi_n}(z)=z_0+z_1a_n+z_2t_n+z_3a_nt_n$, we have
\begin{align*}
  \det(A_{\pi_{n+1}}(z)) & =\det \left(
                          \begin{array}{cc}
                            z_0+z_2a_n & z_1+z_3t_n \\
                            z_1+z_3a_n & z_0+z_2t_n \\
                          \end{array}
                        \right)\\
   &=\det \left( (z_0+z_2a_n)(z_0+z_2t_n)-(z_1+z_3a_n)(z_1+z_3t_n)\right)\\
   &=\det \left((z_0^2-z_1^2)+(z_0z_2-z_1z_3)a_n+(z_0z_2-z_1z_3)t_n+(z_1^2-z_3^2)a_nt_n
        \right)\\
   &=\det\left(A_{\pi_n}(F(z))\right),
\end{align*}
where $F: \mathbb{C}^4\rightarrow \mathbb{C}^4 $ is the homogenous polynomial map defined by
\begin{equation}
  F(z_0,z_1,z_2,z_3)= (~z_0^2-z_1^2~,~z_0z_2-z_1z_3~,~ z_0z_2-z_1z_3~,~z_2^2-z_3^2~).
\end{equation}
It follows that $\det(A_{\pi_{n+1}}(z))=\det\left(A_{\pi_0}(F^{n+1}(z))\right)$, and hence $\det(A_{\pi_{n+1}}(z))=0$ if and only if $F^{n+1}(z)\in L$, where $L$ is the hyperplane $\{z\in \mathbb{C}^4 : z_0+z_1+z_2+z_3=0\}$. So we have $P(A_\pi)\supset \overline{\bigcup_{n\geq 0} F^{-n}(L)}$.

The following lemma gives the extended indeterminacy set of $F$.

\begin{lemma}\label{3.2}
For $F(z)=[~z_0^2-z_1^2~:~z_0z_2-z_1z_3~:~ z_0z_2-z_1z_3~:~z_2^2-z_3^2~]$, the indeterminacy sets
$$I_1= \{ [1:1:\zeta :\zeta], [1:-1:\zeta:-\zeta] ~~|~~ \zeta \in \hat{\mathbb{C}} \},$$
$$I_3=I_2=I_1 \cup \{[1:\zeta :1:\zeta ], [1:\zeta,-1:-\zeta]~|~ \zeta \in \hat{\mathbb{C}}\},$$
 hence the extended indeterminacy set
$$E=I_2=\{ [1:1:\zeta :\zeta], [1:-1:\zeta:-\zeta], [1:\zeta :1:\zeta ] , [1:\zeta,-1:-\zeta] ~~|~~ \zeta \in \hat{\mathbb{C}} \}. $$
\end{lemma}

\begin{proof}
For $z\in I_1$, by definition, $z_0^2-z_1^2=z_0z_2-z_1z_3=z_2^2-z_3^2=0$, hence $z_0=z_1, z_2=z_3$ or $z_0=-z_1,z_2=-z_3$. This is $[1:1:\zeta :\zeta]$ or $[1:-1:\zeta:-\zeta]$.

Since $I_2 = \{z \in \mathds{P}^3  ~|~ F(z) \in I_1 \}$, we must consider pre-images to each
point in $I_1$ with respect to $F$. If $F(z)=[1:1:\zeta :\zeta]$, then $\zeta=1$, hence $z_0^2-z_1^2=z_0z_2-z_1z_3=z_2^2-z_3^2$, and we have $z_0^2-z_1^2+z_2^2-z_3^2=2(z_0z_2-z_1z_3)$, that means $(z_0-z_2)^2=(z_1-z_3)^2$, if $z_0=z_2$, then $z_1=z_3$, that means $z=[1:\zeta :1:\zeta ]$, if $z_0\neq z_2$, then $z_0-z_2=\pm (z_1-z_3)\neq 0$, combine with $z_0^2-z_1^2=z_2^2-z_3^2$, we have $z_0+z_2=\pm (z_1+z_3)$, that implies $z_0=\pm z_1, z_2=\pm z_3$, i.e. $z\in I_1$. If $F(z)=[1:-1:\zeta:-\zeta]$, then $\zeta=-1$, hence $z_0^2-z_1^2=-(z_0z_2-z_1z_3)=z_2^2-z_3^2$, a similar discussion implies that $z\in I_1$ or $z=[1:\zeta:-1:-\zeta]$.

For $I_3$, we only need consider the case $F(z)=[1:\zeta :1:\zeta ]$ or $F(z)=[1:\zeta:-1:-\zeta]$, but this means $\zeta=1$ or $\zeta=-1$, hence $z\in I_2$. So we have $I_3=I_2$, and $E=I_2$.
\end{proof}

\begin{lemma}{\label{lemma3.3}}
Let $z'=(z_0',z_1',z_2',z_3')=F(z)$  , then
\begin{align}
  {z'}_0^2-{z'}_1^2 &=(z_0^2-z_1^2)(z_0^2-z_1^2+z_3^2-z_2^2)-(z_1z_2-z_0z_3)^2  \\
   {z'}_3^2-{z'}_2^2 & =(z_3^2-z_2^2)(z_0^2-z_1^2+z_3^2-z_2^2)-(z_1z_2-z_0z_3)^2\\
   {z'}_0^2-{z'}_1^2+{z'}_3^2-{z'}_2^2 & =(z_0^2-z_1^2+z_3^2-z_2^2)^2-2(z_1z_2-z_0z_3)^2        \\
   z_0'z_2'-z_1'z_3' &=(z_0z_2-z_1z_3)(z_0^2-z_1^2+z_3^2-z_2^2)             \\
   {z'}_1{z'}_2-z_0'z_3' & =(z_1z_2-z_0z_3)^2 .
\end{align}
\end{lemma}
\begin{proof}
This is obtained by a direct computation and by using the algebraic identity :
$$(z_0^2-z_1^2)(z_2^2-z_3^2)=(z_0z_2-z_1z_3)^2-(z_1z_2-z_0z_3)^2.$$
\end{proof}
\begin{remark}
 It is worth noting that we can write $E$ as
\begin{equation*}
  \{~~[z_0:z_1:z_2:z_3] ~|~ z_0z_3-z_1z_2=z_0^2+z_3^2-z_1^2-z_2^2=0  ~~\}.
\end{equation*}
Recall that $\phi$ denotes the cannonical projection from $\mathbb{C}^{n+1}$ to
$\mathbb{P}^n$, and let $S$ be the set $\{z\in \mathbb{C}^4: z_0z_3-z_1z_2=z_0^2+z_3^2-z_1^2-z_2^2=0\}$, then we claim that $E=\phi(S)$. In fact, by Lemma \ref{3.2}, it is easy to check that $E\subset \phi(S)$, for the converse, by Lemma \ref{lemma3.3}, we have $F^2(S)=O$, hence $\phi(S)\subset E$.
\end{remark}

Next we will consider the dynamic properties of  $F$. The main idea is the semi-conjugacy method. More precisely, if we can find a semi-conjugacy function $\varphi: \mathbb{P}^{n}\rightarrow \hat{\mathbb{C}}$ and a polynomial function $f: \hat{\mathbb{C}}\rightarrow \hat{\mathbb{C}}$ such that
\begin{equation*}
  \varphi(F(z_0,z_1,\cdots, z_n))=f(\varphi(z_0,z_1,\cdots,z_n)),
\end{equation*}
then we will have
\begin{equation*}
  \varphi(F^n(z_0,z_1,\cdots, z_n))=f^n(\varphi(z_0,z_1,\cdots,z_n)).
\end{equation*}
Therefore it is suffice to study the  dynamical properties of the one variable function $f$.

\vspace{0.3cm}
For $z\notin S$,  let $x=\frac{z_0^2+z_3^2-z_1^2-z_2^2}{2(z_1z_2-z_0z_3)},~x'=\frac{{z'}_0^2+{z'}_3^2-{z'}_1^2-{z'}_2^2}{2({z'}_1{z'}_2-z_0'z_3')}$ (by the above Remark , $x,x'$ are well defined), then combine the formula $(3.4)$ and $(3.6)$ of Lemma \ref{lemma3.3} , we have
$$x'=2x^2-1.$$
Let $T(z)=2z^2-1$ be the Tchebyshev polynomial,  $\tau(z)=\frac{z_0^2+z_3^2-z_1^2-z_2^2}{2(z_1z_2-z_0z_3)}$ for $z\notin S$,
then we have
\begin{equation}
  \tau(F(z))=T(\tau(z)),~~ ~~ z\notin S.
\end{equation}
Note that $\tau$ is homogeneous of degree $0$, and $\phi(S)=E$, it can be considered as a map from $\mathbb{P}^3\setminus E$ to $\hat{\mathbb{C}}$.  Hence we get the following proposition:
\begin{proposition}\label{prop3.4}
For $k\geq 1$, the following diagram is commutative:
$$\begin{array}[c]{ccc}
   \mathbb{P}^3\setminus E & \stackrel{F^k}{\longrightarrow} & \mathbb{P}^3\setminus E \\
   \downarrow\scriptstyle{\tau} &  & \downarrow\scriptstyle{\tau} \\
   \hat{\mathbb{C}} & \stackrel{T^k}{\longrightarrow} & \hat{\mathbb{C}}
 \end{array}
$$
where $\tau: \mathbb{P}^3\setminus E \rightarrow \hat{\mathbb{C}}$ is defined by $\tau(z)=\frac{z_0^2+z_3^2-z_1^2-z_2^2}{2(z_1z_2-z_0z_3)}$, $F:\mathbb{P}^3\setminus E \rightarrow \mathbb{P}^3\setminus E $ is defined by
\begin{equation}
  F(z_0,z_1,z_2,z_3)= [~z_0^2-z_1^2~:~z_0z_2-z_1z_3~:~ z_0z_2-z_1z_3~:~z_2^2-z_3^2~],
\end{equation}
 and the extended indeterminacy set $E=\{z\in \mathbb{P}^3: z_0z_3-z_1z_2=z_0^2+z_3^2-z_1^2-z_2^2=0\}$.
\end{proposition}

The following proposition gives the explicit form of $F^n$.

\begin{proposition}\label{prop3.5}
Let $F^n_i(z)$ denote the $i$-th component of $F^n(z)$, then for $z\notin\{ z_1z_2-z_0z_3=0\}$,   it holds that
\begin{align}
  F^{n+2}_0(z)&=(z_1z_2-z_0z_3)^{2^{n+1}}\cdot \left( 2^{n+1}\prod_{k=0}^{n} T^k(\tau(z))\right) \left(\frac{z_0^2-z_1^2}{z_1z_2-z_0z_3}- \sum_{k=1}^{n+1} \frac{1}{2^k \prod_{j=0}^{k-1} T^j(\tau(z))} \right), \\
  F^{n+2}_1(z)&=F^{n+2}_2(z)=(z_1z_2-z_0z_3)^{2^{n+1}}\cdot \left( 2^{n+1} \prod_{k=0}^{n} T^k(\tau(z))\right)\cdot \frac{z_0z_2-z_1z_3}{z_1z_2-z_0z_3} ,\\
F^{n+2}_3(z)&=(z_1z_2-z_0z_3)^{2^{n+1}}\cdot \left( 2^{n+1}\prod_{k=0}^{n} T^k(\tau(z))\right) \left(\frac{z_2^2-z_3^2}{z_1z_2-z_0z_3}+ \sum_{k=1}^{n+1} \frac{1}{2^k \prod_{j=0}^{k-1} T^j(\tau(z))} \right).
\end{align}
For $z\in\{ z_1z_2-z_0z_3=0\}$,   it holds that
\begin{align*}
  F^{n+2}_0(z)&=(z_0^2-z_1^2)(z_0^2-z_1^2+z_3^2-z_2^2)^{2^{n+1}-1}, \\
  F^{n+2}_1(z)&=F^{n+2}_2(z)=(z_0z_2-z_1z_3)(z_0^2-z_1^2+z_3^2-z_2^2)^{2^{n+1}-1},\\
F^{n+2}_3(z)&=(z_2^2-z_3^2)(z_0^2-z_1^2+z_3^2-z_2^2)^{2^{n+1}-1}.
\end{align*}
\end{proposition}
\begin{proof}
For convenience, denote $F^n(z)=(F^n_0(z),F^n_1(z),F^n_2(z),F^n_3(z))$ by  $z^{(n)}=(z^{(n)}_0,z^{(n)}_1,z^{(n)}_2,z^{(n)}_3)$. We prove the proposition by induction on $n$. Firstly, it is easy to see that  when $n=0$, the proposition is exactly the Lemma \ref{lemma3.3}. And for the case $z\in\{z\in \mathds{C}^4 ~|~ z_1z_2-z_0z_3=0\}$, by Lemma \ref{lemma3.3},  we have $z^{(n)}_1z^{(n)}_2-z^{(n)}_0z^{(n)}_3=0$ for all $n\geq 0$, then a simple recursive computation yields the conclusion.

For the case $z\notin\{z\in \mathds{C}^4 ~|~ z_1z_2-z_0z_3=0\}$,  by the formula $(3.6)$, we have
$$z^{(n)}_1z^{(n)}_2-z^{(n)}_0z^{(n)}_3=(z_1z_2-z_0z_3)^{2^{n}} ~~~~for~~all~~n\geq 0.$$
Assume the proposition holds for $k=n$, i.e. the formula $(3.9),(3.10),(3.11)$ holds for $k=n$,  then for $k=n+1$, by formula $(3.2)$ and $(3.7)$, we have
\begin{align*}
  z^{(n+1+2)}_0 &=(z_0^{(n+2)})^2-(z_1^{(n+2)})^2 ~~~\\
   &= \left((z^{(n+1)}_0)^2-(z^{(n+1)}_1)^2\right)\left( (z^{(n+1)}_0)^2-(z^{(n+1)}_1)^2+(z^{(n+1)}_3)^2-(z^{(n+1)}_2)^2\right)- \left(z^{(n+1)}_1z^{(n+1)}_2-z^{(n+1)}_0z^{(n+1)}_3  \right)^2\\
   &=z^{(n+2)}_0\left(z^{(n+1)}_1z^{(n+1)}_2-z^{(n+1)}_0z^{(n+1)}_3\right)\cdot 2\tau(z^{(n+1)})-\left(z^{(n+1)}_1z^{(n+1)}_2-z^{(n+1)}_0z^{(n+1)}_3  \right)^2 \\
   &= (z_1z_2-z_0z_3)^{2^{n+1}}\cdot \left(z^{(n+2)}_0\cdot2\tau(z^{(n+1)})-(z_1z_2-z_0z_3)^{2^{n+1}} \right)\\
   &= (z_1z_2-z_0z_3)^{2^{n+1}}\cdot \left(z^{(n+2)}_0 \cdot2T^{n+1}(\tau(z))-(z_1z_2-z_0z_3)^{2^{n+1}} \right).
\end{align*}
By the inductive hypothesis, we have
\begin{align*}
   \frac{z^{(n+1+2)}_0}{(z_1z_2-z_0z_3)^{2^{n+2}}}&=\frac{z^{(n+2)}_0}{(z_1z_2-z_0z_3)^{2^{n+1}}} \cdot2T^{n+1}(\tau(z))-1  \\
   & =\left( 2^{n+1}\prod_{k=0}^{n} T^k(\tau(z))\right) \left(\frac{z_0^2-z_1^2}{z_1z_2-z_0z_3}- \sum_{k=1}^{n+1} \frac{1}{2^k \prod_{j=0}^{k-1} T^j(\tau(z))} \right)\cdot2T^{n+1}(\tau(z))-1\\
   &=\left( 2^{n+2}\prod_{k=0}^{n+1} T^k(\tau(z))\right) \left(\frac{z_0^2-z_1^2}{z_1z_2-z_0z_3}- \sum_{k=1}^{n+1} \frac{1}{2^k \prod_{j=0}^{k-1} T^j(\tau(z))} \right)-1\\
   &=\left( 2^{n+2}\prod_{k=0}^{n+1} T^k(\tau(z))\right) \left(\frac{z_0^2-z_1^2}{z_1z_2-z_0z_3}- \sum_{k=1}^{n+2} \frac{1}{2^k \prod_{j=0}^{k-1} T^j(\tau(z))} \right).
\end{align*}
That is, $(3.9)$ holds for $k=n+1$. By a similar argument about $z_i^{(n+1+2)},i =1,2,3$, we can show that $(3.10),(3.11)$ also hold for $k=n+1$. Hence we complete the proof.
\end{proof}

  Proposition \ref{prop3.5} implies that  a point $z\in F^{-(n+2)}(L)$ if and only if

$$ z_1z_2-z_0z_3\neq0 ~~and~~ (z_0^2-z_1^2+z_2^2-z_3^2+2(z_0z_2-z_1z_3))\prod_{k=0}^n T^k(\tau(z))=0 $$
or
$$ z_1z_2-z_0z_3=0~~and~~ z_0^2-z_1^2+z_3^2-z_2^2=0 .$$
Note that $z_0^2-z_1^2+z_2^2-z_3^2+2(z_0z_2-z_1z_3)=(z_0+z_2+z_1+z_3)(z_0+z_2-z_1-z_3)$  and
$$z_0^2-z_1^2+z_3^2-z_2^2-2(z_1z_2-z_0z_3)=(z_0+z_3+z_1+z_2)(z_0+z_3-z_1-z_2),$$
$$z_0^2-z_1^2+z_3^2-z_2^2+2(z_1z_2-z_0z_3)=(z_0-z_3+z_1-z_2)(z_0-z_3-z_1+z_2),$$
hence the curve $z_0^2-z_1^2+z_2^2-z_3^2+2(z_0z_2-z_1z_3)=0$ is contained in $\bigcup_{x\in [-1,1]}\{z\in \mathbb{C}^4 : z_0^2+z_3^2-z_1^2-z_2^2+2x(z_0z_3-z_1z_2)=0\}$. Since $\bigcup_{n=1}^\infty Z(T^n)$ is dense    in $[-1,1]$, where $Z(T^n)$ denotes the zero set of $T^n$,
 we have
$$\overline{\bigcup_{n\geq 0} F^{-n}(L)}=\bigcup_{x\in [-1,1]}\{z\in \mathbb{C}^4 : z_0^2+z_3^2-z_1^2-z_2^2+2x(z_0z_3-z_1z_2)=0\}= P(A_\pi).$$

\vspace{0.2cm}
Now we are equipped to describe the dynamical  properties of $F$ on  complex projective space.
\begin{theorem}\label{thm3.6}
 For  $D_\infty$ and the associated map $F$ defined in $(3.8)$ , we have $\mathcal{J}(F)=p(A_\pi)$.
\end{theorem}
\begin{proof}
 It is easy to see that $z\in p^c(A_\pi)$ if and only if $\tau(z)\notin [-1,1]$. Assume that $\tau(z)\notin [-1,1]$, then by Lemma \ref{T}, $T^n(\tau(z))\notin[-1,1]$ for all $n\geq 1$.
 By Proposition \ref{prop3.5}, one can write
\begin{equation}
  F^{n+2}(z)=[\frac{z_0^2-z_1^2}{z_1z_2-z_0z_3}-f_n(z): \frac{z_0z_2-z_1z_3}{(z_1z_2-z_0z_3)} :\frac{z_0z_2-z_1z_3}{(z_1z_2-z_0z_3)}:\frac{z_2^2-z_3^2}{z_1z_2-z_0z_3}+f_n ],
\end{equation}
where
\begin{equation}
  f_n(z)=\sum_{k=1}^{n+1} \frac{1}{2^k \prod_{j=0}^{k-1} T^j(\tau(z)) }~, ~n\geq 0~,~z\in p^c(A_\pi).
\end{equation}
Since the iteration sequence $\{T^n \}$ converges normally to $\infty$ on the
Fatou set $\hat{\mathbb{C}}\setminus [-1,1]$, we get that $f_n$ converges normally to some $f$ on $p^c(A_\pi)$, hence $F^n$ converge normally on $p^c(A_\pi)$, that is, $p(A_\pi)\supset \mathcal{J}(F)$.  In fact, if we write $\tau(z)=\cos w$ for some $w\notin \mathbb{R} $, as $T^k(\cos w)=\cos (2^k w)$, then the limit function has the following expression:
\begin{equation}
  f(z)=\sum_{k=1}^\infty \frac{\sin w}{\sin (2^k w)}.
\end{equation}
Moreover, in \cite{BY}, it was proved that $f(z)=\tau(z)-i\sqrt{1-\tau(z)^2}, \tau(z)\notin[-1,1]$.

For the converse,  assume $\tau(z)=\cos w $ for some real number $w$ such that $\frac{w}{\pi}$ is non-dyadic, i.e., $\frac{w}{\pi}$ is not the form of $\frac{q}{2^p}$, then we have
\begin{align*}
  f_n(z) & =\sum_{k=1}^{n+1} \frac{1}{2^k \prod_{j=0}^{k-1} T^j(\tau(z))}=\sum_{k=1}^{n+1} \frac{\sin w}{\sin (2^k w)} \\
   & =\sin w \left(\sum_{k=1}^{n+1} \cot(2^{k-1}w)-\cot(2^k w) \right)\\
   &=\cos w-\sin w \cot(2^{n+1}w).
\end{align*}
Since   $\{~ 2^k w \mod \pi \mid k=0,1,\cdots~\}$ is dense in $[0,\pi]$ and the union of all real numbers $w$ such that $\frac{w}{\pi}$ is non-dyadic is also dense in $\mathbb{R}$, there is no subsequence of $\{\cos w-\sin w \cot(2^{n+1}w)\}$ which can converge on some neighborhood of $w\in \mathbb{R}$. It implies that there is no subsequence of $\{f_n\}$ which can converge on some neighborhood of $z\in \tau^{-1}([-1,1])\setminus E$ since $\tau(z)$ is continuous on $\mathbb{P}^3\setminus E$. Therefore we get that $\{z\in \mathbb{P}^3\setminus E: \tau(z)\in [-1,1]\}$ is contained in $\mathcal{J}(F)$. Combine with the extended indeterminacy set $E$, we have $p(A_\pi)\subset \mathcal{J}(F)$.
\end{proof}

\section{The spectral dynamics of the Lamplighter group $\mathcal{L}$}
\subsection{Some properties of the Chebyshev polynomial of the second kind}

Let $\{U_n\}_{n=0}^\infty$ be the sequence of the $n$-th Chebyshev polynomial of the second kind defined by
\begin{equation}
  U_{n+1}(z)=2z U_n(z)-U_{n-1}(z),~~U_1(z)=2z,~~U_0(z)=1.
\end{equation}
It has the following well-known properties:
\begin{equation}
  U_n(\cos z)=\frac{\sin (n+1)z}{\sin z}, ~~U_n(1)=n+1,~~U_n(-1)=(-1)^n (n+1),
\end{equation}
\begin{equation}\label{L.3}
  U_n(z)=\sum_{k=0}^{[\frac{n}{2}]} (-1)^k \frac{(n-k)!}{k! (n-2k)!} (2z)^{n-2k},
\end{equation}
and $\{U_n(x)\}$ is the orthogonal polynomials on $[-1,1]$ associated to the measure $\sqrt{1-x^2}dx$, that is
\begin{equation}
  \int_{-1}^1 U_m(x) U_n(x) \sqrt{1-x^2} dx =\delta_{mn}.
\end{equation}
By the theory of orthogonal polynomials (for more details see \cite{B.S}), we know that the $U_n$ has exact $n$  simple zeros in the interval $[-1,1]$, and the zeros of $U_n,  U_{n+1}$ strictly interlace.

The following lemma will be useful for us.
\begin{lemma}\label{5.1}
For $z\notin [-1,1]$,
\begin{equation}
  \lim_{n\rightarrow \infty} \frac{U_n(z)}{U_{n+1}(z)}=z-\sqrt{z^2-1}.
\end{equation}
Moreover, if we write $z=\cos w$ with $Im(w)>0 $ , then we can write the above limit property in an elegant way:
\begin{equation}
  \lim_{n\rightarrow \infty} \frac{U_n(\cos w)}{U_{n+1}(\cos w)}=  e^{iw}.
\end{equation}

\end{lemma}
\begin{proof}
Let $w=x+iy, y> 0$, then it is obtained by a direct computation:
\begin{align*}
   \lim_{n\rightarrow \infty} \frac{U_n(\cos w)}{U_{n+1}(\cos w)} &= \lim_{n\rightarrow \infty} \frac{\sin nw}{\sin (n+1)w} \\
   & =\lim_{n\rightarrow \infty} \frac{e^{in(x+iy)}-e^{-in(x+iy)}}{e^{i(n+1)(x+iy)}-e^{-i(n+1)(x+iy)}}\\
&=\lim_{n\rightarrow \infty} \frac{e^{-ny}e^{inx}-e^{ny}e^{-inx}}{e^{-(n+1)y}e^{i(n+1)x}-e^{(n+1)y}e^{-i(n+1)x}}\\
&=e^{i w}.
\end{align*}
\end{proof}

\subsection{The determinant recursion for the Lamplighter group $\mathcal{L}$.  }

For the lamplighter group $\mathcal{L}$, we have the following operator recursions:
$$c=a^{-1}b=b^{-1}a=\left(
                      \begin{array}{cc}
                        0 & I \\
                        I & 0 \\
                      \end{array}
                    \right),
~~a=\left(
      \begin{array}{cc}
        0 & a \\
        b & 0 \\
      \end{array}
    \right),
~~b=\left(
      \begin{array}{cc}
        a & 0 \\
        0 & b \\
      \end{array}
    \right),
$$
$$a^{-1}=\left(
           \begin{array}{cc}
             0 & b^{-1} \\
             a^{-1} & 0 \\
           \end{array}
         \right),
~~b^{-1}=\left(
           \begin{array}{cc}
             a^{-1} & 0 \\
             0 & b^{-1}\\
           \end{array}
         \right).
$$
Hence for the pencil $D_{\pi_n}(z)=z_0+z_1c_n+z_2(a_n+b_n)+z_3(a_n^{-1}+b_n^{-1})$, we have
\begin{align*}
  \det(D_{\pi_{n+1}}(z)) & =\det \left(
                          \begin{array}{cc}
                            z_0+z_2a_n+z_3a_n^{-1} & z_1+z_2a_n+z_3b_n^{-1} \\
                            z_1+z_2b_n+z_3a_n^{-1} & z_0+z_2b_n+z_3b_n^{-1} \\
                          \end{array}
                        \right)\\
  & =\det \left(
                          \begin{array}{cc}
                            z_0-z_1+z_3(a_n^{-1}-b_n^{-1}) & z_1+z_2a_n+z_3b_n^{-1} \\
                            z_1-z_0+z_3(a_n^{-1}-b_n^{-1}) & z_0+z_2b_n+z_3b_n^{-1} \\
                          \end{array}
                        \right)\\
   &=\det \left(
                          \begin{array}{cc}
                            2(z_0-z_1) & z_1-z_0+z_2(a_n-b_n) \\
                            z_1-z_0+z_3(a_n^{-1}-b_n^{-1}) & z_0+z_2b_n+z_3b_n^{-1} \\
                          \end{array}
                        \right)\\
   &=\det\left(2(z_0-z_1)(z_0+z_2b_n+z_3b_n^{-1})-(z_1-z_0+z_3(a_n^{-1}-b_n^{-1}))(z_1-z_0+z_2(a_n-b_n))\right)\\
   &=\det \left((z_0^2-z_1^2-2z_2z_3)+2z_2z_3c_n+z_2(z_0-z_1)(a_n+b_n)+z_3(z_0-z_1)(a_n^{-1}+b_n^{-1})\right)\\
   &=\det\left(D_{\pi_n}(F(z))\right)=\cdots =\det\left(D_{\pi_0}(F^{n+1}(z))\right),
\end{align*}
where  $F: \mathbb{C}^4\rightarrow \mathbb{C}^4 $ is the homogenous polynomial map defined by $$F(z_0,z_1,z_2,z_3)= (~z_0^2-z_1^2-2z_2z_3~,~2z_2z_3~,~ z_2(z_0-z_1)~,~z_3(z_0-z_1)~).$$
Hence $\det(D_{\pi_{n+1}}(z))=0$ if and only if $F^{n+1}(z)\in L$, where $L$ is the hyperplane $\{z\in \mathbb{C}^4 : z_0+z_1+2z_2+2z_3=0\}$. It implies that $P(D_\pi)\supset \overline{\bigcup_{n\geq 0} F^{-n}(L)}$.
\begin{lemma}
\[ \overline{\bigcup_{n\geq 0} F^{-n}(L)}=\overline{\bigcup_{n\geq 0} F^{-n}(\{z_0-z_1=0\})}\bigcup L=L\bigcup \phi^{-1}(E), \]
where $E$ is the indeterminacy set of $F$ and $\phi$ is the cannonical projection from $\mathbb{C}^{4}$ to $\mathbb{P}^3$.
\end{lemma}
\begin{proof}
First, it is easy to see that $F(z)\in L$ if and only if $z_0-z_1=0$ or $z\in L$, that is, $F^{-1}(L)=L\cup \{z\in \mathbb{C}^{4}|z_0-z_1=0\}$. Let $X=\{z\in \mathbb{C}^{4}|z_0-z_1=0\}$, then
\[ F^{-n}(L)=F^{-(n-1)}(L\cup X)=F^{-(n-1)}(X)\cup F^{-(n-1)}(L)=\cdots=\left(\bigcup_{i=0}^{n-1}F^{-i}(X)\right)\bigcup L. \]
Next, note that $F^2(X)=O$, we have $F^{-n}(X)\subset \phi^{-1}(I_{n+2})$, so
\[ \overline{\bigcup_{n\geq 0} F^{-n}(L)}=\overline{\bigcup_{n\geq 0} F^{-n}(X)}\bigcup L \subset \overline{\bigcup_{n\geq 0} \phi^{-1}(I_{n+2})} \bigcup L=L\bigcup \phi^{-1}(E). \]
For the converse, recall that
\[ \phi^{-1}(E)=\overline{\bigcup_{n\geq 0} F^{-n}(O)}. \]
\end{proof}
Indeed, we have proved that
\[ \phi^{-1}(E)=\overline{\bigcup_{n\geq 0} F^{-n}(X)}. \]
For further calculation of the indeterminacy set, let $Q$ be the rational map associated with $F$, this is,
$$Q(z)=[~z_0+z_1-\frac{2z_2z_3}{z_0-z_1}~:~ \frac{2z_2z_3}{z_0-z_1}~:~ z_2~:~z_3~],$$
and for simplicity, write  $z^{(n)}=(z_0^{(n)},z_1^{(n)},z_2^{(n)},z_3^{(n)})=Q^n(z)$.

Since $z_2^{(n)}=z_2, z_3^{(n)}=z_3, z_0^{(n)}+z_1^{(n)}=z_0+z_1$ for all $n$, we have
\begin{equation}\label{4.7}
  z_0^{(n+1)}-z_1^{(n+1)}=z_0+z_1- \frac{4z_2z_3}{z_0^{(n)}-z_1^{(n)}}.
\end{equation}
If we set
\begin{equation}\label{z_n}
  z_0^{(n)}-z_1^{(n)}=\frac{G_n(z)}{H_n(z)}, ~~~~~~~n\geq 0
\end{equation}
with $G_0(z_0,z_1,z_2,z_3)=z_0-z_1, ~H_0(z_0,z_1,z_2,z_3)=1$, then by \eqref{4.7}, we have the following recursive relation:
\begin{equation}\label{L.8}
  \left(
    \begin{array}{c}
      G_{n+1} \\
      H_{n+1} \\
    \end{array}
  \right)
=\left(
   \begin{array}{cc}
     z_0+z_1 & -4z_2z_3 \\
     1 & 0 \\
   \end{array}
 \right)
\left(
    \begin{array}{c}
      G_{n} \\
      H_{n} \\
    \end{array}
  \right), ~~~~~~~n\geq 0.
\end{equation}
As $H_n=G_{n-1},~ n\geq 1$, the sequence ${G_n}$ has the three-term recurrence relation
\begin{equation}
  G_{n+1}(z)=(z_0+z_1)G_n(z)-4z_2z_3G_{n-1}(z).
\end{equation}
The following lemma is inspired by reference \cite{GBS} and the close connection between the orthogonal polynomial and the above three-term recurrence equation.
\begin{lemma}\label{5.2}
For $z\notin \{z_2z_3 = z_0+z_1=0\}$, it holds that
\begin{equation}\label{2.3}
  G_k(z) =(z_0-z_1)(\sqrt{4z_2z_3})^k U_k(\frac{z_0+z_1}{2\sqrt{4z_2z_3}})-(\sqrt{4z_2z_3})^{k+1} U_{k-1}(\frac{z_0+z_1}{2\sqrt{4z_2z_3}})
\end{equation}
for all $k\geq 0$, where we assume that $U_{-1}=0$.
\end{lemma}
\begin{proof}
It is easy to check that it holds for $k=0,1$, by induction, and the recurrence relation of $G_n$ and $U_n$, we have
\begin{align*}
  G_{n+1}(z) &= (z_0+z_1)G_n(z)-4z_2z_3G_{n-1}(z)\\
   & =(z_0+z_1)\left((z_0-z_1)(\sqrt{4z_2z_3})^n U_n(\frac{z_0+z_1}{2\sqrt{4z_2z_3}})-(\sqrt{4z_2z_3})^{n+1} U_{n-1}(\frac{z_0+z_1}{2\sqrt{4z_2z_3}})\right)\\
&-4z_2z_3\left((z_0-z_1)(\sqrt{4z_2z_3})^{n-1} U_{n-1}(\frac{z_0+z_1}{2\sqrt{4z_2z_3}})-(\sqrt{4z_2z_3})^{n} U_{n-2}(\frac{z_0+z_1}{2\sqrt{4z_2z_3}})\right)\\
&=(z_0-z_1)(\sqrt{4z_2z_3})^{n+1}U_{n+1}(\frac{z_0+z_1}{2\sqrt{4z_2z_3}})-(\sqrt{4z_2z_3})^{n+2}U_n(\frac{z_0+z_1}{2\sqrt{4z_2z_3}}),
\end{align*}
which completes the proof.
\end{proof}
It is worth noting that the function $(\sqrt{4z_2z_3})^n U_n(\frac{z_0+z_1}{2\sqrt{4z_2z_3}})$ is independent of the selection of principal value branch of root function. In fact, by (\ref{L.3}), we have
\begin{align*}
  (\sqrt{4z_2z_3})^n U_n(\frac{z_0+z_1}{2\sqrt{4z_2z_3}}) & =(\sqrt{4z_2z_3})^n \sum_{k=0}^{[\frac{n}{2}]} (-1)^k \frac{(n-k)!}{k! (n-2k)!} (\frac{z_0+z_1}{\sqrt{4z_2z_3}})^{n-2k}  \\
   & =\sum_{k=0}^{[\frac{n}{2}]} (-1)^k \frac{(n-k)!}{k! (n-2k)!} (z_0+z_1)^{n-2k}(4z_2z_3)^k.
\end{align*}

\vspace{0.5cm}

Let
\begin{equation}\label{L.11}
  P_n(x,y)=\sum_{k=0}^{[\frac{n}{2}]} (-1)^k \frac{(n-k)!}{k! (n-2k)!} x^{n-2k}y^k=(\sqrt{y})^n U_n(\frac{x}{2\sqrt{y}}),
\end{equation}
then we can write $G_n$ as
\begin{equation}\label{L.12}
  G_n(z)=(z_0-z_1)P_n(z_0+z_1,4z_2z_3)-4z_2z_3P_{n-1}(z_0+z_1, 4z_2z_3).
\end{equation}
The following lemma gives a more precise description of $P_n$.
\begin{lemma}\label{5.3}
\begin{equation}\label{L.13}
  P_n(\alpha+\beta, \alpha \beta)=\frac{\alpha^{n+1}-\beta^{n+1}}{\alpha-\beta},
\end{equation}
Moreover, it implies an identity:
$$\sum_{k=0}^{[\frac{n}{2}]} (-1)^k \left(
                                      \begin{array}{c}
                                        k \\
                                        n-k \\
                                      \end{array}
                                    \right)
 (x+y)^{n-2k}(xy)^k=\frac{x^{n+1}-y^{n+1}}{x-y}.$$
\end{lemma}
\begin{proof}
Let $A=\left(
   \begin{array}{cc}
     z_0+z_1 & -4z_2z_3 \\
     1 & 0 \\
   \end{array}
 \right)$ , then by (\ref{L.8}), we have

\begin{equation*}
  \left(
    \begin{array}{c}
      G_{n} \\
      H_{n} \\
    \end{array}
  \right)
=A^n
\left(
    \begin{array}{c}
      G_{0} \\
      H_{0} \\
    \end{array}
  \right).
\end{equation*}
Let $\alpha, \beta$ be the two eigenvalues of matrix $A$, then $\alpha+\beta=z_0+z_1, \alpha\beta=4z_2z_3$. For the $n$-th power of $2\times 2$ matrix $A$, it has a well known concise form (cf.\cite{J.M}):
\begin{equation*}
  A^n=\alpha^n \left(\frac{A-\beta I}{\alpha-\beta} \right)+\beta^n\left(\frac{A-\alpha I}{\beta-\alpha} \right).
\end{equation*}
Hence we have
\begin{equation*}
  A^n=\left(
   \begin{array}{cc}
     \frac{\alpha^{n+1}-\beta^{n+1}}{\alpha-\beta} & \frac{\alpha \beta^{n+1}-\alpha^{n+1}\beta}{\alpha-\beta}\\
     \frac{\alpha^{n}-\beta^{n}}{\alpha-\beta} & \frac{\alpha \beta^{n}-\alpha^{n}\beta}{\alpha-\beta}\\
   \end{array}
 \right).
\end{equation*}
It shows that
\begin{equation}
  G_n(z)=(z_0-z_1)\frac{\alpha^{n+1}-\beta^{n+1}}{\alpha-\beta}-4z_2z_3\frac{\alpha^{n}-\beta^{n}}{\alpha-\beta}.
\end{equation}
Combining (\ref{L.12}) and the fact that $P_n(\alpha+\beta, \alpha \beta)$ is a homogeneous polynomial of degree of $n$, we get (\ref{L.13}).
\end{proof}

\vspace{0.5cm}
Back to our study of $\phi^{-1}(E)$,  recall that $z^{(n)}=Q^n(z)$. By induction, it is not hard to check that
\[ F^{n}(z)=\prod_{i=0}^{n-1} (z_0^{(i)}-z_1^{(i)})^{2^{n-i}}\cdot Q^n(z). \]
Combining (\ref{z_n}) and the fact $H_n=G_{n-1}$, we have $F^n(z)\in X$ if and only if
\begin{align*}
  0=\prod_{i=0}^n (z_0^{(i)}-z_1^{(i)})^{2^{n-i}}=\prod_{i=0}^n (\frac{G_i}{G_{i-1}})^{2^{n-i}}=\prod_{i=0}^{n-1} G_i^{2^{n-1-i}}.
\end{align*}
Hence we have
\[  \phi^{-1}(E)=\overline{\bigcup_{n\geq 0}\{ z\in \mathbb{C}^4~|~G_n(z)=0 \} }.  \]
Let $\Gamma_n$ be the algebraic variety defined by $G_n$, that is,
\begin{equation}
  \Gamma_n:= \{z\in\mathbb{C}^4 ~  |~ (z_0-z_1)U_n(\tau(z))=\sqrt{4z_2z_3}U_{n-1}(\tau(z))~\},
\end{equation}
where $\tau(z)=\frac{z_0+z_1}{2\sqrt{4z_2z_3}}$, and let $\Gamma$ denote the closure of $\bigcup \Gamma_n$. For the variety $\Gamma_n$, there is a natural question:
\begin{question}
Is  $\Gamma_n$ irreducible for any $n$?
\end{question}

\vspace{0.5cm}

Next we want to have a clear characterization of the set $\Gamma$. It is natural to consider two cases: $\tau(z)\notin [-1,1]$ and $\tau(z)\in [-1,1]$(i.e. $\bigcup_{x\in [0,1]} \{~(z_0+z_1)^2-16xz_2z_3=0 ~\}$).\\

\textbf{Case 1: $\tau(z)\notin [-1,1].$ }

In this case, by Lemma \ref{5.1},
$$\lim_{n\rightarrow \infty} \frac{U_n(\tau(z))}{U_{n+1}(\tau(z))}=\tau(z)-\sqrt{\tau(z)^2-1},$$
it implies that $\Gamma_n$ converges to the critical variety
$$\frac{z_0-z_1}{\sqrt{4z_2z_3}}=\tau(z)-\sqrt{\tau(z)^2-1},$$
that is,
$$(z_0-z_1)z_1=2z_2z_3.$$

\textbf{Case2: $\tau(z)\in [-1,1].$ }

In this case, we will show that $ \{~\tau(z)\in [-1,1]~\} \subset \Gamma$. It is suffice to show that for almost all $x\in \mathds{R}$, the set $ \{\tau(z)=\cos x\}$ is contained $\Gamma$. In particular, we take $x$ as an arbitrary irrational multiple of $\pi$.
Fix $z_2, z_3$, namely $z_2=a, z_3=b$ for some constant $a,b$,  then $z\in \{ \tau(z)=\cos x\}$ if and only if $z_0+z_1=2\sqrt{4z_2z_3}\cos x$ (it is a complex line). For $y\in \mathbb{R}\setminus\{0\}$, define
$$\Gamma_{n,y}:=\Gamma_n \bigcap \{z_2=a, z_3=b, \tau(z)=\cos(x+iy) \}.$$
In fact, it is a single point satisfying $z_0+z_1=2\sqrt{4ab}\cos(x+iy)$ and
$$\frac{2(z_0-z_1)}{z_0+z_1}=\frac{\sin nw}{\cos w \sin(n+1)w}=1-\tan w \cot(n+1)w,$$
where $w=x+iy$.
\begin{lemma}
For any $\epsilon >0$, define the set
$$B_\epsilon:=\bigcup_{n\geq 0, 0<|y|<\epsilon }\Gamma_{n,y},$$
then the complex line $\{ z_0+z_1=2\sqrt{4z_2z_3}\cos x\}$ is contained in the closure of $B_\epsilon$, and hence $\{~\tau(z)\in [-1,1]~\} \subset \Gamma$.
\end{lemma}
\begin{proof}
Since $x$ is an arbitrary irrational multiple of $\pi$, hence the numbers  $\{ nx\mod ~\pi\}$ is dense in $[0, \pi]$, as $y$ is independent of $n$, hence the set $\{\cot(nx+iny), n\geq 0, 0<|y|<\epsilon\}$ is dense in $\mathds{C}$. That means we can chose a point sequence $\{(z_0^{n,y}, z_1^{n,y}, a,b )\in \Gamma_{n,y}\} $ such that $z_0^{n,y}+ z_1^{n,y}$ converge to $z_0+z_1$, and $\frac{z_0^{n,y}- z_1^{n,y}}{z_0^{n,y}+ z_1^{n,y}}$ converge to $\frac{z_0-z_1}{z_0+z_1}$, that is, any point in the complex line can be approximated by the point sequence of $B_\epsilon$, hence the proof is complete.
\end{proof}

\vspace{0.5cm}
By above discussions, we have
\begin{theorem}\label{indet}
\[       \phi^{-1}(E)=\Gamma=\{~(z_0-z_1)z_1=2z_2z_3~\}\bigcup \left( \bigcup_{x\in [0,1]} \{~(z_0+z_1)^2-16xz_2z_3=0 ~\}\right) \bigcup \left( \bigcup_{n=1}^\infty \Gamma_n\right)\]
and
\[ L \cup \Gamma \subset P(D_\pi), \]
 where $L$ is the hyperplane $\{z\in \mathbb{C}^4 : z_0+z_1+2z_2+2z_3=0\}$.
\end{theorem}


\subsection{The dynamical properties of $Q$}
Regard $Q$ as a rational map on $\mathbb{P}^3$ defined by
$$Q(z)=[~z_0+z_1-\frac{2z_2z_3}{z_0-z_1}~:~ \frac{2z_2z_3}{z_0-z_1}~:~ z_2~:~z_3~].$$
Note that by (\ref{4.7}) we can write
$$Q^n(z)=[~\frac{z_0+z_1}{2}+\frac{z_0^{(n)}-z_1^{(n)}}{2}~:~ \frac{z_0+z_1}{2}-\frac{z_0^{(n)}-z_1^{(n)}}{2}~:~ z_2~:~z_3~],$$
hence the dynamic property of $Q$ is determined by $z_0^{(n)}-z_1^{(n)}$, i.e. $\frac{G_n}{G_{n-1}}$ by (\ref{L.8}) and  (\ref{z_n}). Recall that
$$G_n(z_0,z_1)=(z_0-z_1)(\sqrt{4z_2z_3})^n U_n(\tau(z))-(\sqrt{4z_2z_3})^{n+1} U_{n-1}(\tau(z)),$$
where $\tau(z)=\frac{z_0+z_1}{2\sqrt{4z_2z_3}}$,  so we have
\begin{equation*}
  \frac{G_{n+1}}{G_n}=\frac{(\sqrt{4z_2z_3})^{n+2} U_{n+1}\cdot  \left(\frac{z_0-z_1}{\sqrt{4z_2z_3}}-\frac{U_n}{U_{n+1}}\right)}{(\sqrt{4z_2z_3})^{n+1} U_{n} \cdot \left(\frac{z_0-z_1}{\sqrt{4z_2z_3}}-\frac{U_{n-1}}{U_{n}}\right)}.
\end{equation*}
For $\tau(z)\notin [-1,1]$, by Lemma \ref{5.1},  $\frac{G_{n+1}}{G_n}$ converge normal to
$$\sqrt{4z_2z_3})(\tau(z)+\sqrt{\tau(z)^2-1}) =\frac{1}{2}(z_0+z_1+\sqrt{(z_0+z_1)^2-16z_2z_3}),$$
and since $\{\tau(z)\in [-1,1]\}$ is contained in the extended indeterminacy set , we have
\begin{theorem}
For the rational map $Q(z)=[~z_0+z_1-\frac{2z_2z_3}{z_0-z_1}~:~ \frac{2z_2z_3}{z_0-z_1}~:~ z_2~:~z_3~]$,
\[ \mathcal{J}(Q)=E, ~~~\phi(L)\cup E \subset p(D_\pi).  \]
\end{theorem}

In conclusion, we complete the proof of Theorem B, and it is worth noting that $\phi(L)$ is not contained in the extended indeterminacy set $E$. In fact, let $z\in \{z_2=0\}$, by (\ref{L.12}) and (\ref{L.13}), in this case, $G_n(z)=(z_0-z_1)(z_0+z_1)^n$. Then the point $[1:3:0:-2]$ is in $\phi(L)$, but by Theorem \ref{indet}, is not contained in $E$. Hence for the pencil $D_\pi$, the Julia set $\mathcal{J}(Q)$ is strictly contained in $p(D_\pi)$. But maybe the only difference is just $L$, which is dependent of the selection of the generated set of group.
\begin{conjecture}
\[ \phi(L) \cup \mathcal{J}(Q)=p(D_\pi). \]
\end{conjecture}
\begin{acknowledgement}
 The authors would like to thank the referee for his/her careful reading of the paper and helpful suggestions.
\end{acknowledgement}

\end{document}